\DeclareMathAlphabet{\mathpzc}{OT1}{pzc}{m}{it}
\documentclass[headsepline=true]{scrartcl}
\usepackage{amsmath}
\usepackage{amsthm, 	amssymb, amsfonts}
\usepackage[pdfborder={0 0 0}]{hyperref}
\usepackage{amscd}
\usepackage{tensor}
\usepackage{mathrsfs}
\usepackage{arydshln}
\usepackage{tikz}
\usetikzlibrary{decorations.markings}
\usetikzlibrary{decorations.pathreplacing}
\usepackage{lscape}
\usepackage{enumerate}
\usepackage[capitalize]{cleveref}
\usepackage{verbatim}

\DeclareMathOperator{\wt}{wt}

\DeclareMathOperator{\ind}{ind}

\renewcommand{\vec}{\underline}
\renewcommand{\phi}{\varphi}

\newcommand{\n}{\mathbb{N}}
\newcommand{\z}{\mathbb{Z}}

\renewcommand{\c}{\mathbb{C}}

\theoremstyle{plain} 
\newtheorem{thm}{Theorem}[section]
\newtheorem{lem}[thm]{Lemma} 
\newtheorem{lemma}[thm]{Lemma} 
\newtheorem{cor}[thm]{Corollary} 
\newtheorem{prop}[thm]{Proposition}
\newtheorem{conj}[thm]{Conjecture}

\theoremstyle{definition}

\newtheorem{example}[thm]{Example}

\theoremstyle{remark}
\newtheorem{remarkx}{Remark}
\newenvironment{remark}
  {\pushQED{\qed}\remarkx}
  {\popQED\endremarkx}

\newcommand{\BC}{{\mathbb{C}}}

\newcommand{\BH}{{\mathbb{H}}}

\newcommand{\BQ}{{\mathbb{Q}}}

\newcommand{\BZ}{{\mathbb{Z}}}

\newcommand{\CF}{{\mathcal F}}

\newcommand{\CR}{{\mathcal R}}

\newcommand{\Fz}{{\mathfrak{z}}}

\newcommand{\QMod}{\mathsf{QMod}}
\newcommand{\Mod}{\mathsf{Mod}}
\newcommand{\QJac}{\mathsf{QJac}}

\newcommand{\A}{{\mathsf{A}}}

\newcommand{\p}{\mathbb{P}}

\newcommand{\Mbar}{{\overline M}}
\newcommand{\DR}{\mathrm{DR}}
\newcommand{\ev}{{\mathrm{ev}}}

\title{Gromov--Witten theory of K3 surfaces and a Kaneko--Zagier equation for Jacobi forms}
\author{Jan-Willem van Ittersum%
\thanks{\emph{Email}: \href{mailto:j.w.m.vanittersum@uu.nl}{j.w.m.vanittersum@uu.nl}, \newline
Mathematisch Instituut, Universiteit Utrecht, Postbus 80.010, 3508 TA Utrecht, The Netherlands, \newline
Max-Planck-Institut f\"ur Mathematik, Vivatsgasse 7, 53111 Bonn, Germany
}, Georg Oberdieck\thanks{\emph{Email}: \href{mailto:georgo@math.uni-bonn.de}{georgo@math.uni-bonn.de}, \newline
Mathematisches Institut, Universit\"at Bonn, Endenicher Allee 60, 53115 Bonn, Germany
}, Aaron Pixton\thanks{\emph{Email}: \href{mailto:pixton@umich.edu}{pixton@umich.edu}, \newline
Department of Mathematics, University of Michigan, Ann Arbor, USA
}}

\begin{document}
\maketitle


\begin{abstract}
We prove the existence of quasi-Jacobi form solutions for an analogue of the Kaneko--Zagier differential equation for Jacobi forms.
The transformation properties of the solutions under the Jacobi group are derived.
A special feature of the solutions is the polynomial dependence of the index parameter.
The results yield an explicit conjectural description for all double ramification cycle integrals in the Gromov--Witten theory of K3 surfaces.
\end{abstract}

\section{Introduction}
\subsection{K3 surfaces}
The Yau--Zaslow formula
(proven by Beauville \cite{B} and Bryan--Leung \cite{BL}) 
evaluates the generating series
of counts of rational curves on K3 surfaces in primitive classes
as the inverse 
of the discriminant
\[ \Delta(\tau) = q \prod_{n \geq 1} (1-q^{n})^{24} \]
where~$q = e^{2 \pi i \tau}$ and~$\tau \in \BH$ is the standard variable of the upper half-plane.

More general curve counts on K3 surfaces are defined by the Gromov--Witten invariants
\[
\Big\langle \alpha ; \gamma_1, \ldots, \gamma_n \Big\rangle^S_{g,\beta}
:=
\int_{[\Mbar_{g,n}(S,\beta)]^{\text{red}}}
\pi^{\ast}(\alpha) \prod_{i=1}^{n} \ev_i^{\ast}(\gamma_i)
\]
where $\Mbar_{g,n}(S,\beta)$ is the moduli space of~$n$-marked genus~$g$ stable maps to a K3 surface~$S$ representing the class~$\beta \in H^2(S,\BZ)$, and
\[ \pi : \Mbar_{g,n}(S,\beta) \to \Mbar_{g,n}, \quad \ev_i : \Mbar_{g,n}(S,\beta) \to S, \ i=1,\ldots, n \]
are the forgetful and evaluation maps.
The integral is taken over the
reduced virtual fundamental class and the insertions are arbitrary classes
\[ \alpha \in H^{\ast}(\Mbar_{g,n}), \quad \gamma_1, \ldots, \gamma_n \in H^{\ast}(S). \]

Let~$\mathsf{a} = (\mathsf{a}_1, \ldots, \mathsf{a}_n)$ be a list of integers with~$\sum_i \mathsf{a}_i = 0$. The moduli space
$
\Mbar_{g,n}(\p^1,\mathsf{a})^{\widetilde{}},
$
defined in relative Gromov--Witten theory, parametrizes stable maps from a curve of genus~$g$ to~$\p^1$ with ramification profiles over~$0$ and~$\infty$ given respectively by the positive and negative entries in~$\mathsf{a}$.
The double ramification cycle
\[ \DR_{g}(\mathsf{a}) \in H^{2g}(\Mbar_{g,n}) \]
is defined as the pushforward under the forgetful map~$\Mbar_{g,n}(\p^1,\mathsf{a})^{\widetilde{}} \to \Mbar_{g,n}$ of the virtual class on this moduli space (see \cite{DR}).

Let also~$z \in \BC$ and~$p = e^{z}$, and 
consider the odd (renormalized) Jacobi theta function
\[
\Theta(z,\tau) 
=  (p^{1/2}-p^{-1/2})\prod_{m\geq 1} \frac{(1-pq^m)(1-p^{-1}q^m)}{(1-q^m)^{2}}.
\]

The following formula was found in the study of the quantum cohomology
of the Hilbert scheme of points of a K3 surface in \cite{HilbK3},
and related to K3 surfaces in \cite{K3xE}.

\begin{conj}[ \cite{HilbK3, K3xE} ] \label{conj:K3DR} There exist quasi-Jacobi forms~$\varphi_{m}(z,\tau), \phi_{m,n}(z,\tau)$ such that
for all primitive effective~$\beta \in H^2(S,\BZ)$ we have 
\begin{multline*}
\sum_{g = 0}^{\infty} \Big\langle \DR_g(\mathsf{a}) ; \gamma_1, \ldots, \gamma_n \Big\rangle^S_{g,\beta} (-1)^{g+n} z^{2g-2+n} \\
=
\frac{1}{\prod_i \mathsf{a}_i^{\deg(\gamma_i)}}
\mathrm{Coeff}_{q^{\frac{1}{2} \beta^2}}\left( 
\sum_{\{ (a_j, b_j) \}_j, \{ c_j \}_{j} }
\frac{1}{\Theta^2 \Delta}
\prod_{j} (\gamma_{a_j}, \gamma_{b_j}) \varphi_{a_j b_j} \cdot \prod_{j} (\gamma_{c_j}, \beta) \varphi_{c_j}  \right).
\end{multline*}
Here, the sum on the right side is over all partitions of the set~$\{ (\mathsf{a}_i, \gamma_i) \}_{i=1}^{n}$
into parts of size~$\leq 2$. The parts of size~$1$ are labeled by~$(c_j, \gamma_{c_j})$, and the parts of size~$2$ are labeled
$\{ (a_j, \gamma_{a_j}), (b_j, \gamma_{b_j}) \}$. Moreover,
$\deg(\gamma)$ denotes half the cohomological degree of~$\gamma$, i.e.~$\gamma \in H^{2 \deg(\gamma)}(S)$,
and
$( - , - )$ is the Mukai pairing on~$H^{\ast}(S)$ defined by
\[ \big( (r_1, D_1, n_1) , (r_2, D_2, n_2) \big) = r_1 n_2 + r_2 n_1 - D_1 \cdot D_2 \]
where we write ~$D_1 \cdot D_2 = \int_S D_1 \cup D_2$ for the intersection of divisors.
\end{conj}

We refer to Section~\ref{subsec:QuasiJacobi} for the definition of quasi-Jacobi forms.
The left hand side of the conjecture is a (virtual) count of curves on K3 surfaces,
whose normalization admits a map to~$\p^1$ with prescribed ramification over two points of the target and with the ramification points
incident to given cycles~$\gamma_i$.
If there are no marked points, the double ramification cycle is the top Chern class~$\lambda_g$ of the Hodge bundle over the moduli spae of curves, 
\[ \DR_{g}( \emptyset ) = (-1)^g \lambda_g. \]
In this case the conjecture specializes to the Katz--Klemm--Vafa formula
\[
\sum_{g = 0}^{\infty} \left\langle \lambda_g \right\rangle^{S}_{g,\beta} z^{2g-2} = \mathrm{Coeff}_{q^{\frac{1}{2} \beta^{2}}}
\left( \frac{1}{\Theta(z,\tau)^2 \Delta(\tau)} \right).
\]
proven in \cite{MPT}. 
Further evidence for the conjecture has been obtained in \cite{K3xP1}.

While the functions~$\varphi_m, \varphi_{m,n}$ were conjectured to be quasi-Jacobi forms (of explicit weight and index)
they have been left indeterminate in \cite{HilbK3, K3xE}. The goal of this paper
is simply to give an explicit formula for these functions and study their properties.

\subsection{A Kaneko--Zagier equation for Jacobi forms}
Let~$D_\tau=\frac{1}{2\pi i}\frac{d}{d\tau} = q\frac{d}{dq}$
and consider the ratio
$$F(z):=\frac{D_{\tau}^2\Theta(z)}{\Theta(z)} = -\sum_{n \geq 1}\sum_{d | n} (n/d)^3(p^{d/2}-p^{-d/2})^2 q^{n},$$
where, as we will often do, have dropped~$\tau$ from the argument.

We define formal series~$\phi_m \in \BQ[p^{\pm \frac{1}{2}}][[q]]$ for all~$m \in \BZ$ by the differential equation
\begin{align}
\label{eq:dif} D_{\tau}^2\phi_m = m^2F\phi_m, 
\end{align}
together with the constant term
\begin{equation} \phi_m = (p^{m/2}-p^{-m/2})+O(q). \label{eq:cst term} \end{equation}
Since the constant term of~$F$ in~$q$ vanishes, \eqref{eq:dif} determines the functions~$\phi_m$ uniquely from the initial data.
By definition, we have~$\phi_{-m} = -\phi_m$.

Our first main result is the following characterization of the functions~$\phi_m$.

\begin{thm} \label{thm:sol}
For all~$m \geq 0$ we have
$$\phi_m = \mathrm{Res}_{x=0} \left(\frac{\Theta(x+z)}{\Theta(x)}\right)^m.$$
In particular,~$\phi_m$ is a quasi-Jacobi form of weight~$-1$ and index~$|m|/2$ for every~$m$.
\end{thm}
\vspace{7pt}

Consider the ratio of theta functions
\[ f(x) = \frac{\Theta(x+z)}{\Theta(x)} \]
whose appearance in mathematics goes back to work of Eisenstein \cite{Weil}.
Since its inverse has Taylor expansion~$1/f(x) = \Theta(z)^{-1} x + O(x^2)$,
the function~$1/f(x)$
can be formally inverted.
By Lagrange inversion, Theorem~\ref{thm:sol} then precisely says that the inverse series is the generating series of the~$\phi_m$:
\begin{equation} y = \frac{1}{f(x)} \ \Longleftrightarrow \ x = \sum_{m = 1}^{\infty} \frac{\phi_m}{m} y^m. \label{inversion} \end{equation}

Let us explain the connection of the differential equation \eqref{eq:dif} to a well-known differential equation for modular forms.
Recall the Eisenstein series
\[ E_k(\tau) = 1 - \frac{2k}{B_{k}} \sum_{n \geq 1} \sum_{d|n} d^{k-1} q^n, \]
where the weight~$k \geq 2$ is even and~$B_k$ are the Bernoulli numbers. Let
\[ \vartheta_k = D_{\tau} - \frac{k}{12} E_2(\tau) 
\]
be the Serre derivative which restricts to an operator
$\Mod_k \to \Mod_{k+2}$
on the space of modular forms of weight~$k$.
The Kaneko--Zagier equation \cite{KZ} is the differential equation
\begin{equation} \vartheta_{k+2} \vartheta_k f_k  = \frac{k (k+2)}{144} E_4(\tau) f_k. \label{KZdiff} \end{equation}
If~$k\equiv 0$ or~$4$ mod~$6$ it has non-trivial solutions which are modular forms of weight~$k$.
A direct calculation shows that a function~$f_k$ is a solution to \eqref{KZdiff}
if and only if~$g_{k+1} = f_k/\eta^{2k+2}$,
with~$\eta(\tau) = q^{1/24} \prod_{n \geq 1} (1-q^n)$ the Dedekind function,
is a solution of
\[ D_{\tau}^2 g_m = m^2 \frac{E_4(\tau)}{144} g_{m}. \]

We observe that the differential equation \eqref{eq:dif}
is a Jacobi-form analogue of the Kaneko--Zagier equation.
Even stronger, since \eqref{eq:dif} does not involve derivatives in the elliptic variable
we can specialize it to~$\frac{z}{2 \pi i}=a \tau + b$ for any~$a,b \in \BQ$ and in this way obtain an infinite family of Kaneko--Zagier type
differential equations with quasi-modular solutions\footnote{The solutions~$\phi_m$ restrict to modular forms (for a congruence subgroup) at~$\frac{z}{2 \pi i}=1/2$}.
The inversion formula \eqref{inversion} has the classical analogue \cite[Thm.~5(iv)]{KZ}
\[ x = \sum_{k \geq 1} \frac{f_{k-1}}{k} y^k \ \Longleftrightarrow \ y = \left( \frac{\wp'(x)}{-2} \right)^{-1/3}, \]
where the role of~$f(x)$ is played by the formal cube root of the derivative
$\wp'(x) = \frac{d}{dx} \wp(x)$ of the Weierstrass elliptic function~$\wp(x)$,
and the solutions~$f_k$ are normalized accordingly. 

We refer to Section~\ref{sec:KZequations} for a general construction of differential equations of Kaneko--Zagier type.

\subsection{Differential equation of the second kind}
We are also interested in a second family of functions, defined in terms of the~$\phi_m$ of the previous section.

Define formal series~$\phi_{m,n} \in \BQ[p^{\pm 1/2}][[q]]$ for all~$m,n \in \BZ$ by the differential equation
\begin{equation} D_{\tau} \phi_{m,n} = mn\phi_m\phi_n F+(D_{\tau}\phi_m)(D_{\tau}\phi_n) \label{defining_diff_eqn} \end{equation}
together with the condition that the constant term vanishes:
\[ \phi_{m,n} = O(q). \]
Since~$\phi_m$ is odd in~$m$, the definition implies the symmetries
\[ \forall m,n : \ \phi_{m,n} = \phi_{n,m} = \phi_{-m,-n}. \]
Moreover,~$\phi_{m,0} = 0$ as~$\phi_0=0$.
Our second main result describes the modular properties of~$\phi_{m,n}$:

\begin{thm} \label{Thm_phi_mn} For all~$m,n \in \BZ$ the difference
\[ \phi_{m,n} - |n| \delta_{m+n,0} \]
is a quasi-Jacobi form of weight~$0$ and index~$\frac{1}{2}( |m| + |n| )$.
\end{thm}

If~$m \neq -n$ the proof of Theorem~\ref{Thm_phi_mn} is easy. Indeed, in this case we have
\begin{equation} \phi_{m,n} = \frac{m}{m+n}\phi_mD_{\tau}(\phi_n) + \frac{n}{m+n}D_{\tau}(\phi_m)\phi_n \label{phimn_prop} \end{equation}
and since the algebra of quasi-Jacobi forms is closed under differentiation with respect to both~$z$ and~$\tau$
the result follows from Theorem~\ref{thm:sol}.
It hence remains to consider the case $m=-n$. However, since the algebra of quasi-Jacobi forms is not closed under integration,
this case is not obvious at all.

A key feature of the functions~$\phi_m$ is their polynomial dependence on~$m$. Precisely,
their Taylor expansion in the elliptic variable is of the form
\[ \phi_m = \sum_{k \geq 1} P_k(m) z^k \]
where each~$P_k$ is a polynomial in~$m$ of degree~$\leq k$ with coefficients quasi-modular forms.
This implies that the~$\phi_{m,n}$ depend polynomially on~$m,n$ as well.
Hence we are allowed to take the limit of the formula~\eqref{phimn_prop}. The result is
$$\phi_{n,-n} = D_\tau(\phi_n)\phi_{-n}+n(D_\tau(\phi'_{-n})\phi_n - \phi'_{-n}D_\tau\phi_n),$$
where~$\phi'_u$ is the formal derivative of~$\phi_u$ with respect to~$u$.
But, by inspection the function~$\phi'_n$ is usually not a quasi-Jacobi
and hence from this point it is still unclear why~$\phi_{n,-n}$ should be quasi-Jacobi.
Instead our proof of Theorem~\ref{Thm_phi_mn} relies on a subtle interplay between
holomorphic anomaly equations,
which measure the defect of~$\phi_m$ and~$\phi_{m,n}$ to be honest Jacobi forms, and the aforementioned polynomiality.

The holomorphic anomaly equations we derive are also of independent interest
since they determine the precise transformation behaviour of the
functions~$\varphi_{m}$ and~$\varphi_{m,n}$ under the Jacobi group.
As another indirect consequence of the proof of Theorem~\ref{Thm_phi_mn} we obtain a third, recursive characterization of the function~$\phi_m$:
\begin{prop} \label{prop:recursion}
For all~$m,n \geq 1$ we have
\[
\phi_{m+n}=\frac{1}{m}D_z(\phi_m)\phi_n+\frac{1}{n}\phi_mD_z(\phi_n)+\sum_{i+j=m}\frac{1}{i}\phi_{i,n}\phi_j + \sum_{i+j=n}\frac{1}{i}\phi_{i,m}\phi_{j}.
\]
\end{prop}

We finally relate the functions~$\varphi_m$ and~$\varphi_{m,n}$ to the geometry of K3 surfaces.

\begin{conj}
The functions~$\varphi_{m}$ and~$\varphi_{m,n}$ as defined above are the functions appearing in Conjecture~\ref{conj:K3DR}.
\end{conj}

Besides plenty of numerical evidence which is known for \cref{conj:K3DR}, there are several qualitative features of the functions~$\phi$
which correspond to similar features in Gromov--Witten theory.
The polynomial dependence of the~$\phi$'s is reflected in the polynomial
dependence of the double ramification cycle on the ramification profiles \cite{DR}.
In their Taylor expansions the~$z$-coefficients of the~$\phi$'s are quasi-modular forms.
This matches a result of \cite{MPT}.
The quasi-Jacobi form property and the holomorphic anomaly equations are expected from holomorphic-symplectic geometry \cite{HilbK3} and the results of \cite{HAE}.

Conjecture~\ref{conj:K3DR} yields an explicit formula for the Gromov--Witten theory of~$K3 \times \p^1$ relative to two fibers over~$\p^1$.
In terms of this theory, efficient algorithms to determine the Gromov--Witten invariants of all CHL Calabi--Yau threefolds are known \cite{BO-CHL}. This leads to deep relations between counting on K3 surfaces and Conway moonshine.
We hope to come back to these questions in future work.

\subsection{Acknowledgements}
We thank A.~Oblomkov for pointing out the polynomial dependence in the~$\varphi_m$,
and R.~Pandharipande for 
the many things we learned from him about
the Gromov--Witten theory of K3 surfaces.

The second author was partially funded by the DFG grant OB 512/1-1. The third author was partially supported by NSF grant 1807079.

\section{Preliminaries}
\subsection{Quasi-modular forms}
For all even~$k > 0$ consider the renormalized Eisenstein series
\[ G_k(\tau) = - \frac{B_k}{2 \cdot k} + \sum_{n \geq 1} \sum_{d|n} d^{k-1} q^n. \]

The~$\BC$-algebras~$\Mod = \oplus_k \Mod_k$ and~$\QMod = \oplus_k \QMod_k$
of modular and quasi-modular forms can be described by Eisenstein series:
\[ \Mod = \BQ[G_4, G_6], \quad \QMod = \BQ[G_2, G_4, G_6]. \]
The algebra~$\QMod$ is acted on by both~$D_{\tau} = q \frac{d}{dq}$
and the operator~$\frac{d}{dG_2}$ which takes the formal derivative in~$G_2$
when a quasi-modular forms is written as a polynomial in~$G_2, G_4, G_6$.
Let also~$\wt$ be the operator on~$\QMod$ that acts on~$\QMod_k$ by multiplication by~$k$.
We have the~$\mathrm{sl}_2$-commutation relation
\[ \left[ \frac{d}{dG_2}, D_{\tau} \right] = -2 \wt. \]

\subsection{Theta functions}
Let~$z \in \BC$ and~$p = e^{z}$. Let
\[
\vartheta_1(z,\tau) = \sum_{\nu\in \mathbb{Z}+\frac{1}{2}} (-1)^{\lfloor \nu \rfloor} p^\nu q^{\nu^2/2}
\]
be the odd Jacobi theta function.\footnote{The Jacobi function~$\vartheta_1$ defines
the unique section on the elliptic curve~$\BC_{w}/(\BZ+\tau \BZ)$ which vanishes at the origin.
In our convention the variable~$w$ of the complex plane~$\BC_w$ is related to~$z$ by~$z=2 \pi i w$.
In other words, the fundamental region of the curve is given by~$\frac{z}{2\pi i} \in \{ a + b \tau \mid a,b \in [0,1] \}$. 
} 
By the Jacobi triple product we have
\[ \Theta(z) = \vartheta_1(z,\tau) / \eta^3(\tau). \]
The product formula for~$\Theta$ yields also the expansion
\begin{align}
\Theta(z) 
&=  z \exp\left(-2\sum_{k\geq 2} G_k \frac{z^k}{k!}\right).
\label{theta_z_expansion}
\end{align}

\subsection{Quasi-Jacobi forms} \label{subsec:QuasiJacobi}
Jacobi forms are a generalization of classical modular forms
which depend on an elliptic parameter~$z \in \BC$ and a modular parameter~$\tau \in \BH$, see \cite{EZ} for an introduction.
Quasi-Jacobi forms are constant terms of almost holomorphic Jacobi forms.
Following \cite{Lib} and \cite[Sec.1]{OPix2} we shortly recall the definition.

Consider the real-analytic functions
\[ \nu = \frac{1}{8 \pi \Im(\tau)}, \quad \alpha = \frac{\Im(z/2\pi i)}{\Im(\tau)}. \]
An almost holomorphic function on~$\BC \times \BH$ is a function of the form
\[ \Psi = \sum_{i, j \geq 0} \psi_{i,j}(z,\tau) \nu^i \alpha^j \]
such that each of the finitely many non-zero~$\psi_{i,j}$ is holomorphic
and admits a Fourier expansion of the form~$\sum_{n \geq 0} \sum_{r \in \BZ} c(n,r)q^n p^r$ in the region~$|q|<1$.
An almost holomorphic weak Jacobi form of weight~$k$ and index~$m \in \BZ$
is an almost holomorphic function on~$\BC\times \BH$ which satisfies the transformations laws of Jacobi forms of this weight and index \cite{EZ}.
A quasi-Jacobi form of weight~$k$ and index~$m$ is a function~$\psi(z,\tau)$
such that there exists an almost holomorphic weak Jacobi form~$\sum_{i,j} \psi_{i,j} \nu^i \alpha^j$ with~$\psi_{0,0} = \psi$.

In this paper we will also work with quasi-Jacobi forms of half-integral index $\frac{m}{2} \in \frac{1}{2} \BZ$.
These are defined identical as above except that we include (in the usual way) a character in the required transformation law.
The character we use for index $m/2$ is defined by the transformation properties of $\Theta^m(z)$ under the Jacobi group.\footnote{This character is essentially uniquely determined by requiring that the square of a half-integral weight Jacobi form is a Jacobi form without character, see for example the discussion in \cite{GSZ}.}
In particular, $\Theta(z)$ is a (quasi) Jacobi form of weight~$-1$ and index~$1/2$;
its square $\Theta(z)^2$ is a Jacobi form without character.

The algebra of quasi-Jacobi forms is bigraded by weight~$k$ and index~$m$:
\[ \QJac = \bigoplus_{k} \bigoplus_{m \in \frac{1}{2}\BZ} \QJac_{k,m}. \]
In index~$0$ we recover the algebra of quasi-modular forms:~$\QJac_{k,0} = \QMod_k$.

Similar to the case of quasi-modular forms, the algebra of quasi-Jacobi forms can be embedded in a polynomial algebra.
Let
$D_z = \frac{d}{dz} = p \frac{d}{dp}$
and consider the series
\begin{align*}
\A(z) = \frac{D_z \Theta(z)}{\Theta(z)}
& = - \frac{1}{2} - \sum_{m \neq 0} \frac{p^m}{1-q^m}
\end{align*}
and the Weierstra{\ss} elliptic function
\[
\wp(z,\tau) = \frac{1}{12} + \frac{p}{(1-p)^2} + \sum_{d \geq 1} \sum_{k | d} k (p^k - 2 + p^{-k}) q^{d}.
\]
We write~$\wp'(z,\tau) = D_z \wp(z,\tau)$ for its derivative with respect to~$z$.
Since taking the derivative with respect to~$z$ and~$\tau$ preserves the algebra of quasi-Jacobi forms (\cite{OPix2})
it is easy to see that all of these are (meromorphic) quasi-Jacobi forms.

\begin{prop} 
The algebra~$\CR = \BC[\Theta, \A, G_2, \wp, \wp', G_4]$ is a free polynomial ring,
and~$\QJac$ is equal to the subring of all polynomials which define holomorphic functions~$\BC \times \BH \to \BH$.
\end{prop}
\begin{proof}
It is immediate that if~$f \in \CR$ is holomorphic, then it is a quasi-Jacobi form.
Conversely, divide any quasi-Jacobi form of index~$m/2$ by~$\Theta^{m}$.
The result then follows from \cite[Sec.\ 2]{Lib}.
\end{proof}

\begin{remark}
The algebra~$\CR$ is the algebra of all meromorphic quasi-Jacobi forms with the property that all poles are at the lattice points $z=m+n\tau$ with $m,n\in \z$. Indeed, since~$\Theta, \A$ and $G_2$ lie in~$\CR$, it suffices to show that meromorphic Jacobi forms of index~$0$ with the latter property are elements of~$\CR$. For such a Jacobi form there exists a polynomial in $\wp$ and $\wp'$ with modular coefficients such that the sum is holomorphic and elliptic, hence constant. Therefore, every such meromorphic Jacobi form lies in~$\CR$. \end{remark}

The weight and index of the generators of~$\CR$ are given as follows:
$$\begin{array}{lll} 
\text{Generator} & \text{Weight} & \text{Index} \\\hline
\Theta & -1 & 1/2 \\
\A & 1  & 0\\
G_2 & 2 & 0\\
\wp & 2 & 0\\
\wp' & 3 & 0\\
G_4 & 4 & 0.
\end{array}$$

Consider the formal derivative operators~$\frac{d}{d\A}$ and~$\frac{d}{d G_2}$.
Let~$\mathrm{wt}$ and~$\mathrm{ind}$ be the operators which act on~$\QJac_{k,m}$ by multiplication by the weight~$k$ and the index~$m$ respectively.
By \cite[(12)]{OPix2} we have the commutation relations:
\begin{equation} \label{eq:comm relations 1}
\begin{alignedat}{2}
\left[ \frac{d}{dG_2}, D_{\tau} \right] & = -2 \mathrm{wt}, & \qquad \qquad \left[\frac{d}{d \A}, D_z \right] & = 2 \mathrm{ind} \\
\left[\frac{d}{dG_2}, D_z \right] & = -2\frac{d}{dA}, &\qquad  \left[\frac{d}{dA}, D_{\tau}\right] & = D_z.
\end{alignedat}
\end{equation}

The almost-holomorphic Jacobi forms completing~$\A$ and~$G_2$ are given by
\begin{equation} \widehat{\A} = \A + \alpha, \quad \widehat{G_2} = G_2 + \nu. \label{AhatG2hat} \end{equation}
Moreover all other generators of $\CR$ are (meromorphic) Jacobi forms.
Hence the formal derivatives~$\frac{d}{d \A}$ and~$\frac{d}{d G_2}$ of a quasi-Jacobi form measure
the dependence of its completion on the non-holomorphic variables~$\alpha$ and~$\nu$,
or in other words the failure of a quasi-Jacobi forms to be an honest Jacobi forms.
For a quasi-Jacobi form we call~$\frac{d}{dA} \psi$ its holomorphic anomaly.
An equation of the form ($\frac{d}{dA} \psi = \ldots$) will be called a holomorphic anomaly equation.
Similar definitions apply to~$\frac{d}{dG_2}$.

As explained in \cite{OPix2} knowing the holomorphic-anomaly equations of a quasi-Jacobi form
is equivalent to knowing their transformation properties unter the Jacobi group.
Concretely, we have the following (the case of half-integral index is similar):
\begin{lemma}[ \cite{OPix2} ]
Let~$\psi(z,\tau) \in \QJac_{k,m}$ with~$m \in \BZ$. Then
\begin{gather*}
\psi\left( \frac{z}{c \tau + d}, \frac{a \tau + b}{c \tau + d} \right) 
= (c \tau + d)^{k} e\left( \frac{c m (z/2 \pi i)^2}{c \tau + d} \right)
\exp\left( - \frac{c \frac{d}{dG_2}}{4 \pi i (c \tau + d)}  + \frac{c \frac{z}{2 \pi i} \frac{d}{dA}}{c \tau+d} \right) \psi(z,\tau) \\
\psi(z + 2 \pi i(\lambda \tau + \mu), \tau)
= e\left( - \lambda^t L \lambda \tau - 2 \lambda^t L \frac{z}{2 \pi i} \right) \exp\left( - \lambda \frac{d}{dA} \right) \psi(z,\tau).
\end{gather*}
\end{lemma}

\subsection{Multivariate quasi-Jacobi forms}
As in \cite[Sec.\ 2]{OPix2} one can similarly define quasi-Jacobi forms of rank~$n$ in which the dependence of the variable~$z \in \BC$
is generalized to a dependence on the vector
\[ z = (z_1, \dots, z_n) \in \BC^n. \]
The index of quasi-Jacobi forms of rank~$n$ is given by a symmetric matrix
\[
m
=
\begin{pmatrix}
m_{11} & \cdots & m_{1n} \\
\vdots & \ddots & \vdots \\
m_{n1} & \cdots & n_{nn}
\end{pmatrix}.
\]
Although a description of the algebra~$\QJac^{(n)}$ of rank~$n$ quasi-Jacobi forms
in terms of concrete polynomial rings is not available in general,
using the expansions \eqref{AhatG2hat} shows that we have an embedding
\[ \QJac^{(n)} \subset \mathsf{MJac}^{(n)}\left[ G_2, \A(z_1), \ldots, \A(z_n) \right] \]
where we let~$\mathsf{MJac}^{(n)}$ denote the algebra of meromorphic-Jacobi forms of rank~$n$.
In particular, the formal derivative operators
\[ \frac{d}{d\A(z_i)}, \quad \frac{d}{d G_2} \]
are well-defined.\footnote{See also \cite{OPix2} for a direct definition via the almost-holomorphic completions.} By \cite[(12)]{OPix2} the operators satisfy the commutation relations
\begin{equation} \label{eq:comm relations 2}
\begin{alignedat}{2}
\left[ \frac{d}{dG_2}, D_{\tau} \right] & = -2 \mathrm{wt}, & \qquad \qquad \left[\frac{d}{dA(z_i)}, D_{z_j} \right] & = 2 \mathrm{ind}_{i,j}, \\
\left[\frac{d}{dG_2}, D_{z_i} \right] & = -2\frac{d}{d \A(z_i)}, &\qquad  \left[\frac{d}{dA(z_i)}, D_{\tau}\right] & = D_{z_i},
\end{alignedat}
\end{equation}
where the operator~$\mathrm{ind}_{i,j}$ multiplies a quasi-Jacobi form of index~$m$ by~$m_{ij}$.

\subsection{Polynomiality}
The following simple lemma about polynomials will be convenient for us later.
\begin{lemma} \label{Lemma_polynomiality}
Let~$f(u,v)$ be a polynomial in variables~$u,v$ and let~$F(u)$ be the unique polynomial such that
$\forall n \geq 1: \, F(n) = \sum_{j=0}^{n-1} f(j, n-j)$. Then
\[ F(-n) = - \sum_{j=1}^{n} f(-j, -n+j). \]
\end{lemma}
\begin{proof}
For all~$m\in\mathbb{Z}, n > 0$ define
\[
G(m,n) = \sum_{j=0}^{n-1} f(j, m-j).
\]
This agrees with a unique polynomial~$P(m,n)$. Now extend~$G$ to all~$m,n\in\mathbb{Z}$ by setting~$G(m,0) = 0$ and
\[
G(m,n) = -\sum_{j=1}^{-n} G(-j, m+j)
\]
for all~$m\in\mathbb{Z},n < 0$. We then have that~$G(m,n+1)-G(m,n) = f(n,m-n)$ is a polynomial for all~$m,n$. But~$P(m,n+1)-P(m,n)$ is also a polynomial. The two polynomials agree for~$n > 0$, so they agree for all~$n$; since~$G$ and~$P$ also agree for~$n > 0$, this means that they must also agree for all~$n$.

The lemma now follows, since it is just saying that~$F(-n) = P(-n,-n) = G(-n,-n)$.
\end{proof}

We also will find the following language convenient: we say that a set of power series~$f_m(z) \in R[[z]], m \in \BZ$ for some coefficient ring~$R$ is \emph{polynomial in~$m$} if there exist
polynomials~$P_k(u) \in R[u]$ such that
\[ \forall m \in \BZ : \ f_m(z) = \sum_{k \geq 0} P_k(m) z^k. \]
In our case the coefficient ring~$R$ will usually be the ring of quasi-modular forms~$\QMod$.

\section{Differential equation}
In this section we study the function~$\phi_m$ defined by the differential equation \eqref{eq:dif} and
the constant term~$\phi_m = p^{m/2} - p^{-m/2} + O(q)$.
We first prove the evaluation
$$\phi_m = \mathrm{Res}_{x=0} \left(\frac{\Theta(x+z)}{\Theta(x)}\right)^m$$
which immediately implies that~$\phi_m$ is a quasi-Jacobi form.
We then study the Fourier expansion of~$\phi_m$,
discuss the dependence of~$\phi_m$ on the parameter~$m$,
and derive a holomorphic anomaly equation.

\subsection{Proof of \cref{thm:sol}} \label{subsection:proof_gen_series}
\emph{Define} functions~$\phi_m$,~$m \geq 0$ by the claim of the theorem i.e. let
$\phi_m = \mathrm{Res}_{x=0} \left(\frac{\Theta(x+z)}{\Theta(x)}\right)^m$.
We need to check that these function satisfy the differential equations \eqref{eq:dif} and have the right constant term \eqref{eq:cst term}.
Checking the constant term is straightforward and we omit the details (see also Section~\ref{subsection_Fourier_expansion}).
To check the differential equation we form the generating series
$g(y) = \sum_{m \geq 1} y^m \phi_m / m$. 
Let also~$D_y = y \frac{d}{dy}$.
The differential equation (\ref{eq:dif}) is then equivalent to
\begin{equation} D_{\tau}^2 g(y) = F(z,\tau) D_y^2 g(y). \label{eqn:diff for g} \end{equation}

Consider the function
\[ f(x) = \frac{\Theta(x+z)}{\Theta(x)}. \]
We will apply the variable change
\[ y = \frac{1}{f(x)}\ \Longleftrightarrow\ x = g(y) \]
where we have used Lagrange inversion to identify the inverse of~$1/f$ with the generating series~$g(y)$.
Let~$f'(x) := D_x f := \frac{d}{dx} f(x)$. 
By differentiating~$f(g(y))=y$ and applying the chain rule we find the transformations
\begin{gather*}
D_y g(y) = -\frac{f}{f'}, \quad \quad 
D_{\tau} g(y) = -\frac{D_\tau f}{f'}, \quad \quad 
D_y^2 g(y) = -\frac{f}{f'} \cdot \frac{f'' f - (f')^2}{(f')^2}, \\
D_{\tau}^2 g(y) = -\frac{1}{(f')^3} \left[ D_{\tau}^2(f) (f')^2 - 2 f' \cdot D_{\tau}(f) D_{\tau}(f') + f'' \cdot D_{\tau}(f)^2 \right].
\end{gather*}
Applying these and changing variables the differential equation \eqref{eqn:diff for g} becomes 
\begin{equation} \label{diffeq}
D_x(f)^2 D_{\tau}^2(f) -  2  D_x(f)  D_x D_{\tau}(f)  D_{\tau}(f) + D_x^2(f)  D_{\tau}(f)^2 = F(z,\tau) \cdot D_x^2 \log(f) \cdot f^3.
\end{equation}

The functions~$\Theta(x+z)$ and~$\Theta(x)$ are Jacobi forms of rank~$2$ in the elliptic variables~$(x,z)$ of index~$\frac{1}{2} \binom{1\ 1}{1\ 1}$ and~$\binom{1/2\ 0}{0\phantom{/2}\ 0}$ respectively.
Hence~$f(x)$ is a Jacobi form of weight 0 and index
\[ \begin{pmatrix} 0 & 1/2 \\ 1/2 & 1/2 \end{pmatrix}. \]
We need to show that the following function vanishes:
\[ \CF(x,z) = D_x(f)^2 D_{\tau}^2(f) - 2 D_x(f) D_x D_{\tau}(f) D_{\tau}(f) + D_x^2(f) D_{\tau}(f)^2 - F(z,\tau) \cdot D_x^2 \log(f) \cdot f^3. \]
As a polynomial in the derivatives of~$f$, the function~$\CF$ is a rank~$2$ quasi-Jacobi form
of weight~$6$ and index~$\frac{1}{2} \binom{0\ 3}{3\ 3}$. 
Using the commutation relations \eqref{eq:comm relations 2} a direct check shows
\[
\frac{d}{dG_2} \CF = \frac{d}{dA(x)} \CF = 0.
\]
In particular, by \cite[Lem.\ 6]{OPix2}
we have~$\CF(x+ 2 \pi i \tau, z) = p^{-3} \CF(x,z)$.
Moreover, by considering the Taylor expansion one checks (e.g. using a computer\footnote{The code for this computation as well as a parallel computation in Section~\ref{sec:holomorphic anomaly equations} can be found on the webpage of the second author. It also contains functions which express the $\phi_m, \phi_{mn}$ in terms of the generators of $\CR$.}) that~$\CF$ is holomorphic at~$x=0$ and
vanishes to order~$3$ at~$x=-z$ (use the variable change~$\tilde{x} = x+z$).
We conclude that the ratio
$\CF / f^3$ is a double-periodic and holomorphic in~$x$, so a constant in~$x$.
The constant is a quasi-Jacobi form in~$z$ and is easily checked to vanish. 
This shows that the differential equation is satisfied.
The claim that the~$\varphi_m$ are quasi-Jacobi forms of the specified weight follows from Lemma~\ref{cor_main_thm} below.
\qed

\medskip
Define the operator on the algebra of quasi-Jacobi forms by
\[ D=D_z + 2 G_2 \frac{d}{d \A} \]
We conclude the following structure result.

\begin{lemma}\label{conj:2} \label{cor_main_thm}
For every~$m \geq 0$ there exist modular forms~$h_k \in \Mod_{m - k -1}$ such that
\[ \phi_m = \sum_{k = 0}^{m-1} h_k(\tau) \cdot D^k(\Theta(z)^m). \]
Hence every~$\phi_m$ is a quasi-Jacobi form of weight~$-1$ and index~$\frac{|m|}{2}$, and~$\frac{d}{dG_2} \phi_m = 0$.
\end{lemma}
\begin{proof}[Proof of Corollary]
For any power series~$f(z)$ we have
\[ e^{D_z x}f(z)=f(x+z). \]
Moreover, the Baker--Campbell--Hausdorf formula 
and the relations \eqref{eq:comm relations 1} yield
\begin{equation} \label{BCH}
e^{D_z x} e^{2 G_2 \frac{d}{d \A} x} = e^{x D - 2 x^2 G_2 \ind} = e^{-2 x^2 G_2 \ind} e^{xD}.
\end{equation}
We find that
\begin{align}
\frac{\Theta(x+z)^m }{\Theta(x)^m}
& = \Theta(x)^{-m} e^{D_z x} \left( \Theta(z)^m \right) \notag \\
& = \Theta(x)^{-m} e^{D_z x} e^{2 G_2 \frac{d}{d \A} x} \left( \Theta(z)^m \right) \notag \\
& \overset{\eqref{BCH}}{=} \Theta(x)^{-m} e^{- m x^2 G_2} e^{xD} \left(  \Theta(z)^m \right) \notag \\
& = x^{-m} \exp\left( 2m \sum\nolimits_{k \geq 4}G_k \frac{x^k}{k!} \right) e^{xD} \left(  \Theta(z)^m \right) \label{eq:resfm}
\end{align}
where we used \eqref{theta_z_expansion} in the last step.
Taking the coefficient of~$x^{-1}$ yields the first claim.
The second claim follows from the commutation relation~$[ \frac{d}{dG_2}, D ] = 0$. 
\end{proof}

\begin{remark}
For all $m \geq 0$ we have
\[ 
\phi_{-m}
= \mathrm{Res}_{x=-z} \left(\frac{\Theta(x+z)}{\Theta(x)}\right)^{-m}
\]
Indeed, after the variable change~$x' = -(x+z)$ the right hand side becomes
\[ - \mathrm{Res}_{x'=0} \left( \frac{\Theta(-x'-z)}{\Theta(-x')} \right)^m = - \phi_m. \qedhere \]
\end{remark}

\subsection{Fourier expansion} \label{subsection_Fourier_expansion}
By integrating the function
\[ f_m(x) = \left( \frac{\Theta(x+z)}{\Theta(x)} \right)^m. \]
around the sides of a fundamental region and using~$f_m(x+\tau,z) = p^{-m} f_m(x,z)$ 
one gets
\[ \varphi_m = \mathrm{Res}_{x=0} f_m = (1-p^{-m}) \mathrm{Coeff}_{\sigma^0} f_m(x,z,\tau) \]
where~$\sigma = e^{x}$ is the Fourier variable associated to~$x$.\footnote{See also \cite[App.~A]{HAE} for a similar argument.}

An application of the Jacobi triple product and computing the power by~$m$ by taking first the log of each product term, multiplying it by~$m$ and then exponentiating again,
together with a bit of reordering the terms, then yields from this the expression
\begin{align*} 
\varphi_m
& = (p^{m/2} - p^{-m/2}) \mathrm{Coeff}_{\sigma^0} \exp\left( \sum_{k \neq 0} \frac{m}{k} \sigma^k \frac{1-p^k}{1-q^k} \right) \\
& = (p^{m/2} - p^{-m/2}) \sum_{|a| = 0} \left( \prod_i \frac{1-p^{a_i}}{1-q^{a_i}} \right) \frac{m^{l(a)}}{\Fz(a)}
\end{align*}
where the sum in the second equation is over all generalized partitions with non-zero parts summing up to~$0$.
Moreover, if we write~$a = (i^{a_i})_{i \in \BZ\setminus \{0\}}$ then
$\Fz(a) = \prod_i i^{a_i} a_i!$ is the standard automorphism factor.
The first Fourier coefficients of~$\phi_m$ are
\[ \varphi_{m} = (s^{m} - s^{-m}) \left( 1 - m^2 (s-s^{-1})^2 q + O(q^2) \right) \]
where we have written~$s = e^{z/2}$ so~$p=s^2$.

\subsection{The solution $\phi_m$ as a function of $m$}
In this section we consider~$\phi_m$ as a function of~$m$ viewed as a (formal) variable.
To distinguish with the case~$m \in \BZ$ we will replace~$m$ by a variable~$u$.

We give three different formulas for~$\phi_u$. First, consider the expansion
\[ F(s,q) = \sum_{k \geq 1} F_k(s) q^k, \quad \quad  F_k(s) = -\sum_{d|k} \Big( \frac{k}{d} \Big)^3 (s^{d} - s^{-d})^2 \]
where as before we have used~$s = e^{z/2}$ so~$p=s^{2}$.
Then by an immediate check the differential equation \eqref{eq:dif} for~$\varphi_m$ is equivalent to the following formula:
\begin{equation} \label{phiu_formula}
\phi_u = 
(p^{u/2} - p^{-u/2})
\left( 1 + \sum_{m \geq 1} \sum_{k_1, \ldots , k_m \geq 1} 
\frac{ F_{k_1}(s) F_{k_2}(s) \cdots F_{k_m}(s) }{ k_1^2 (k_1+k_2)^2 \ldots (k_1+...+k_m)^2 } q^{k_1 + \ldots + k_m} u^{2m} \right)
\end{equation}

Second we can use the Fourier expansion of the~$\phi_m$ as discussed in Section~\ref{subsection_Fourier_expansion}:
\[
\phi_u
=
(p^{u/2} - p^{-u/2})
\sum_{|a| = 0} \left( \prod_i \frac{1-p^{a_i}}{1-q^{a_i}} \right) \frac{u^{l(a)}}{\Fz(a)}
\]

We see that Theorem~\ref{thm:sol} is equivalent to the following non-trivial identity:
\begin{multline*} \sum_{|a| = 0} \left( \prod_i \frac{1-p^{a_i}}{1-q^{a_i}} \right) \frac{u^{l(a)}}{\Fz(a)} 
=
1 + 
\sum_{\substack{m \geq 1 \\ k_1, \ldots , k_m \geq 1}} 
\frac{ F_{k_1}(s) F_{k_2}(s) \cdots F_{k_m}(s) }{ k_1^2 (k_1+k_2)^2 \ldots (k_1+...+k_m)^2 } q^{k_1 + \ldots + k_m} u^{2m}.
\end{multline*}

For the third formula, we use a Taylor expansion in~$u$.
For positive integers~$u$ one can write the solution
$$\mathrm{Coeff}_{x^{-1}} \left(\frac{\Theta(x+z)}{\Theta(x)}\right)^u$$
as
\begin{align}
\varphi_u
&=\mathrm{Coeff}_{x^{-1}} \frac{(x+z)^u}{x^u}\exp\left(2u\sum\nolimits_{k\geq 2} G_k \frac{x^k-(x+z)^k}{k!}\right) \notag \\
&=\mathrm{Coeff}_{x^{-1}} \sum_{\ell=1}^\infty \binom{u}{\ell}\left(\frac{z}{x}\right)^\ell\exp\left(2u\sum\nolimits_{k\geq 2} G_k \frac{x^k-(x+z)^k}{k!}\right). \label{u-exp}
\end{align}
The latter expression makes sense as an element of~$\c[[z]]$ for all~$u\in \c$. 
For example, the first terms read
\small
 \begin{align*}
 \varphi_u & = uz -  G_{2} u^{3}z^{3} + \left(\left(\frac{1}{3} G_{2}^{2} - \frac{1}{72} G_{4}\right) u^{5} 
 + \left(\frac{1}{6} G_{2}^{2} - \frac{5}{72} G_{4}\right) u^{3}\right)z^{5} + O(z^7).
 \end{align*}
\normalsize

The expansion \eqref{u-exp} yields the following important structure result.
\begin{prop} \label{prop_polynomiality} For every~$k \geq 1$ there exist odd polynomials~$P_{k}(u)$ of degree~$\leq k$ with coefficients in~$\QMod_{k-1}$ such that
for all~$m \in \BZ$
\[ \phi_{m} = \sum_{\text{odd } k \geq 1} z^k P_{k}(m). \]
Moreover,~$P_1(u) = u$ and if~$k \geq 2$, then~$u^3 \mid P_k(u)$. 
\end{prop}

\subsection{Anomaly equation}
We consider the holomorphic anomaly of~$\phi_m$ with respect to the variable~$z$.
\begin{prop}\label{prop:anomaly}
For all~$m\geq 1$ one has
\[ \frac{d}{d \A}\varphi_m = \frac{1}{2}\sum_{\substack{i+j=m\\ i,j \geq 1}} \frac{m^2}{ij}\varphi_i \varphi_j.\]
\end{prop}

It follows that every~$z^k$ coefficient of~$\frac{d}{dA} \phi_m$ is polynomial in~$m$ in the range~$m\geq 0$.
However the dependence on~$m$ is only piecewise polynomial in general:

\begin{cor} \label{Cor_polynomiality_Anomaly}
The difference
\[
\phi_m^A = \frac{d}{dA} \phi_m - m z \phi_m \delta_{m < 0}
\]
depends polynomially on~$m$, i.e.\ there exist polynomials~$Q_k(u)$ of degree~$\leq k+1$ with coefficients in~$\QMod_{k-2}$ such that
$\phi_m^A = \sum_{k \geq 2} z^k Q_{k}(m)$.
Moreover,~$u^2 \mid Q_k$ for all~$k$.
\end{cor}

\begin{proof}[Proof of Corollary~\ref{Cor_polynomiality_Anomaly}]
We first rewrite the proposition as 
\[ \frac{d}{dA} \phi_m = m \sum_{j=1}^{m-1} \phi_j \cdot \frac{\phi_{m-j}}{m-j} \]
Hence for all~$m \geq 0$ we have
$\frac{d}{dA} \phi_m = \sum_n Q_n(m) z^n$
where the polynomials $Q_n$ are determined by
\[ Q_n(m) = m \sum_{\substack{k + \ell = n \\ k, \ell \geq 1}} \sum_{j=1}^{m-1} P_k(j) \frac{P_{\ell}(m-j)}{m-j} \]
for all $m \geq 0$.
Here~$P_k(m)$ are the polynomials of Proposition~\ref{prop_polynomiality}.

For all $m>0$ by Lemma~\ref{Lemma_polynomiality} we have
\begin{align*} Q_n(-m) 
& = -(-m) \sum_{\substack{k + \ell = n \\ k, \ell \geq 1}} \sum_{j=1}^{m} P_k(-j) \left( \frac{P_{\ell}(u)}{u} \right)\Big|_{u=-m+j}  \\
& = -m \sum_{\substack{k + \ell = n \\ k, \ell \geq 1}} \sum_{j=1}^{m-1} P_k(j) \frac{P_{\ell}(-m+j)}{-m+j} - m P_{n-1}(m),
\end{align*} 
where we used the second part of Proposition~\ref{prop_polynomiality} for the last equality.
Summing up we obtain as desired
\[ \phi_{-m}^A = \sum_{n} z^n Q_n(-m) = -\frac{d}{dA} \phi_m - m z \phi_m. \qedhere \]
\end{proof}

\begin{proof}[Proof of \cref{prop:anomaly}]
We give first a proof via generating series. As in the proof of \cref{thm:sol} consider the generating series
\[ g(y) = \sum_{m \geq 1} \frac{\phi_m}{m} y^m \]
and let~$D_y = y \frac{d}{dy}$.
We need to prove the equality
\[ \frac{d}{d \A(z)} g(y) = g(y) D_y g(y). \]

Let~$f(x) = \frac{\Theta(x+z)}{\Theta(x)}$ so that
$f(g(y)) = \frac{1}{y}$. Then by~$[ \frac{d}{d \A}, D_z ]=2 \ind$ we have
\begin{equation} \label{aaa}
\frac{d}{d \A(z)} f(x) = \frac{d}{d \A(z)} \frac{e^{D_z x} \Theta(z)}{\Theta(x)} = 
\frac{[\frac{d}{d \A(z)}, D_z] x e^{D_z x} \Theta(z)}{\Theta(x)} = 
x \cdot f(x).
\end{equation}
Applying~$\frac{d}{dA}$ to~$f(g(y))=1/y$ we get
$(\frac{d}{dA}f)(g(y)) + (D_x f)(g(y)) \frac{d}{dA} g(y) = 0$,
and hence
\[ \frac{d}{d\A(z)} g(y) = - \frac{g(y)}{y \cdot (D_x f)}. \]
Since we also have
\[ D_y( f(g(y)) ) = (D_x f)(g(y)) D_y(g(y)) = - \frac{1}{y}, \text{ and hence } \frac{1}{D_x f} = -y \cdot D_y g(y) \]
the claim follows.
\end{proof}

We give a more direct proof of \cref{prop:anomaly} using the following combinatorial Lemma whose proof follows directly from Lagrange inversion
and is left to the reader. 
\begin{lem}\label{lem:powerseries} Let~$f(x)$ be a power series and~$k\in \n$. Then for all~$m \geq 1$ we have
\[\frac{1}{m}\cdot [ f(x)^m ]_{x^{m-k}} = \frac{1}{k}\sum_{n_1+\ldots+n_k=m}\prod_{i=1}^k \frac{1}{n_i}[ f(x)^{n_i} ]_{x^{{n_i}-1}}\]
where we write~$[ - ]_{x^m}$ for taking the coefficient of~$x^m$.
\end{lem}

\begin{proof}[Second proof of \cref{prop:anomaly}]
Observe that by \eqref{aaa} we have
$$\frac{d}{d A} \left[ \left(\frac{\Theta(x+z)}{\Theta(x)}\right)^m \right]_{x^{-1}} = m \left[ \left(\frac{\Theta(x+z)}{\Theta(x)}\right)^m \right]_{x^{-2}},$$
Applying \cref{lem:powerseries} with~$k=2$ and~$f=x\frac{\Theta(x+z)}{\Theta(x)}$ 
yields the desired result. 
\end{proof}

\section{Differential equation of the second kind}
Recall the two defining properties of the series~$\phi_{m,n}$:
\begin{itemize}
\item the differential equation:~$\displaystyle D_{\tau} \phi_{m,n} = mn\phi_m\phi_n F+(D_{\tau}\phi_m)(D_{\tau}\phi_n)$
\item the vanishing of the constant term:~$\displaystyle  \phi_{m,n} = O(q)$.
\end{itemize}
The goal of this section is to first prove that~$\phi_{m,n}$ are quasi-Jacobi forms (Theorem~\ref{Thm_phi_mn}),
and then derive their holomorphic anomaly equations (\cref{subsec:HAE2}).
\subsection{Polynomiality}
We first recall the following.
\begin{prop}\label{prop:phimn}
If~$m\neq -n$ then we have
$$\phi_{m,n} = \frac{m}{m+n}\phi_mD_{\tau}(\phi_n) + \frac{n}{m+n}D_{\tau}(\phi_m)\phi_n.$$
\end{prop}
\begin{proof}
The differential equation follows from the defining differential equation \eqref{eq:dif} satisfied by~$\phi_m$. The vanishing of the constant term is observed directly.
\end{proof}

By definition and the polynomiality of~$\phi_m$ the series~$\phi_{m,n}$ is a power series in~$z$ and~$q$ with coefficients which are polynomials in~$m$ and~$n$.
We use Proposition~\ref{prop:phimn} to prove a stronger statement.

\begin{prop} \label{prop:polymn}
There exist polynomials~$P_{r}(u,v)$ of degree at most~$r$ in variables~$u,v$ with coefficients quasi-modular forms of weight~$r$ such that for all~$m,n \in \BZ$
\[ \phi_{m,n} = \sum_{r>0} z^r P_r(m,n). \]
Moreover, the polynomials~$P_{r}(u,v)$ are divisible by both~$u^2$ and~$v^2$.
\end{prop}

\begin{proof}
By the defining differential equation \eqref{defining_diff_eqn} and the polynomiality of~$\phi_m$ 
there exist polynomials~$P_{a,r}(u,v)$ of degree~$r+2$ with rational coefficients such that
\[ \phi_{m,n} = \sum_{r>0} z^r \sum_{a \geq 1} q^a P_{a,r}(m,n) \]
for all~$m,n \in \BZ$. Here we have~$r>0$ since~$\phi_m(z=0) = 0$ for all~$m$.

On the other hand by Proposition~\ref{prop:phimn} for all~$m,n \in \BZ$ with~$m \neq -n$ we have
\[ \phi_{m,n} = \sum_{r>0} z^r \frac{1}{m+n} \sum_{k+\ell = r} \Big( n D_{\tau}(P_k(m)) P_{\ell}(n) + m P_k(m) D_{\tau}(P_{\ell}(n)) \Big) \]
where~$P_k(u)$ are the polynomials of Proposition~\ref{prop_polynomiality}.
Since the inner sum vanishes when setting~$m=-n$ and it is polynomial of degree at most~$r+1$ in~$m,n$,
there exists a polynomial~$P_{r}(u,v)$ of degree at most~$r$ with coefficients in~$\QMod_{r}$ such that
\[ \phi_{m,n} = \sum_{r>0} z^r P_r(m,n) \]
whenever~$m \neq -n$.

The equality of polynomials
\[ \sum_{a \geq 1} q^a P_{a,r}(u,v) = P_r(u,v) \]
holds after evaluating~$(u,v)$ at~$(m,n)$ for all integers~$m \neq -n$. Hence the equality holds as an equality of polynomials.

The last statement follows 
since~$n D_{\tau}(P_k(m)) P_{\ell}(n) + m P_k(m) D_{\tau}(P_{\ell}(n))$ is divisible by both~$m^2$ and~$n^2$,
hence the same holds for the term obtained by dividing by~$m+n$.
\end{proof}

\begin{example} The first terms in the Fourier and Taylor expansions of~$\phi_{m,n}$ are
\begin{align*}
\varphi_{m,n} & = 
- mn (s^{m} - s^{-m}) (s^{n} - s^{-n}) (s-s^{-1})^2 q + O(q^2)
\end{align*}
where~$s = e^{z/2}$, and
\small
\begin{multline*}
\phi_{u,v} = \left((2 G_{2}^{2} - \frac{5}{6} G_{4}) u^{2} v^{2}\right)z^{4}
+ \Big( (-\frac{4}{3} G_{2}^{3} + \frac{2}{3} G_{2} G_{4} - \frac{7}{720} G_{6}) (u^{4} v^{2} + u^2 v^4) \\
+ (-\frac{2}{3} G_{2}^{3} + \frac{1}{6} G_{2} G_{4} + \frac{7}{720} G_{6}) u^{3} v^{3} 
+ (-\frac{2}{3} G_{2}^{3} + \frac{5}{6} G_{2} G_{4} - \frac{7}{144} G_{6}) u^{2} v^{2}\Big)z^{6} + O(z^{7}).
\end{multline*}
\normalsize
\end{example}


\subsection{Holomorphic anomaly equations} \label{sec:holomorphic anomaly equations}
From Proposition~\ref{prop:phimn} we can deduce for all~$m \neq -n$
the following anomaly equation:
\begin{equation} \label{dAphi}
\begin{aligned}
\frac{d}{dA} \phi_{m,n}
= \ \ 
& \frac{n}{m+n} \left( D_z(\phi_m) \phi_n + D_{\tau}( \frac{d}{dA} \phi_m) \phi_n + D_{\tau}(\phi_m) \cdot \frac{d}{dA} \phi_n \right) \\
+ & \frac{m}{m+n} \left( \phi_m D_{z}(\phi_n) + (\frac{d}{dA}\phi_m) \cdot D_{\tau}(\phi_n) + \phi_m \cdot D_{\tau} \frac{d}{dA} \phi_n \right)
\end{aligned}
\end{equation}
By the anomaly equation for~$\phi_m$ this gives an expression for~$\frac{d}{dA} \phi_{m,n}$ whenever~$m \neq -n$.

In case~$m,n > 0$ we can find a more efficient equation:

\begin{prop} \label{HAE_mnpositive}
For all~$m,n > 0$,
\begin{equation} \label{dAphi2}
\frac{d}{dA} \varphi_{m,n}
=
\frac{m \cdot n}{m+n} \varphi_{m+n} + \sum_{j=1}^{m-1} \frac{m}{j} \varphi_{m-j,n} \varphi_j
+\sum_{j=1}^{n-1} \frac{n}{j} \varphi_{m,n-j} \varphi_j
\end{equation}
\end{prop}

\begin{proof}
We prove that the right hand side in \eqref{dAphi2} is equal to the right hand side in \eqref{dAphi}.
By the anomaly equation for~$\phi_m$ and comparing terms it is equivalent to prove the following equation for all~$m,n \geq 1$:
\begin{align}\label{eq:rec2}\phi_{m+n}=\frac{1}{m}D_z(\phi_m)\phi_n+\frac{1}{n}\phi_mD_z(\phi_n)+\sum_{i+j=m}\frac{1}{i}\phi_{i,n}\phi_j + \sum_{i+j=n}\frac{1}{i}\phi_{i,m}\phi_{j}.\end{align}
We multiply both sides with~$x^m y^n$ and sum over all~$m,n \geq 1$. With~$g(x) = \sum_{m \geq 1} x^m \phi_m/m$ the equation becomes
\begin{multline} \label{13445}
\frac{y D_x g(x) - x D_y g(y)}{x-y} = D_z g(x) \cdot D_y g(y) + D_x g(x) \cdot D_z g(y) \\
+ \left( (D_x + D_y)^{-1} D_y h(x,y) \right) D_x g(x) + \left( (D_x + D_y)^{-1} D_x h(x,y) \right) D_y g(y)
\end{multline}
where~$(D_x + D_y)^{-1}$ acts term-wise by multiplying the coefficient of~$x^m y^n$ by~$(m+n)^{-1}$ (this is well defined since both~$m,n$ are positive for all non-zero coefficients)
and we have used
$(D_x + D_y) \sum_{m, n \geq 1} \frac{\phi_{m,n}}{m} x^m y^n = D_y h(x,y)$ 
with
\[ h(x,y) = D_x g(x) \cdot D_{\tau} g(y) + D_{\tau} g(x) \cdot D_y g(y). \]

Rewriting~$D_y = (D_x + D_y) - D_x$ we have 
\[ (D_x + D_y)^{-1} D_y h = h - (D_x + D_y)^{-1} D_x h. \]
Inserting this the~$(D_x + D_y)^{-1}$ term factors out and we obtain that \eqref{13445} is equivalent to
\begin{multline*}
D_x h = 
(D_x + D_y) \Bigg( \frac{1}{ D_y g(y) - D_x g(x)} \\
\times \left( \frac{y D_x g(x) - x D_y g(y)}{x-y} - D_z g(x) \cdot D_y g(y) + D_x g(x) \cdot D_z g(y) \right) \Bigg)
\end{multline*}
Expanding and using that~$(D_x + D_y)( y/(x-y) ) = 0$ this is equivalent to
\begin{multline} \label{5dsad}
\left( D_x^2 g(x) \cdot D_y g(y) - D_x g(x) \cdot D_y^2 g(y) \right) \cdot (1 + D_z g(x) + D_z g(y) + h ) \\
+ (D_y g(y) - D_x g(x)) \cdot \Big( D_z\big( D_x g(x) \cdot D_y g(y) \big) + D_x g(x) D_y (h) + D_y g(y) D_x h \Big) = 0.
\end{multline}

We consider again the function
\[ f(x) = \frac{\Theta(x+z)}{\Theta(x)} \]
and apply the variable change
\[ x = \frac{1}{f(\tilde{x})},\ y = \frac{1}{f(\tilde{y})} \quad \Longleftrightarrow \quad \tilde{x} = g(x),\ \tilde{y} = g(y). \]
Let us denote~$f'(x) = \frac{d}{dx} f(x)$. We then have the transformations
\begin{alignat*}{2}
D_x g(x) & = -\frac{f}{f'} & \quad \quad 
D_z g(x) & = -\frac{D_z f}{f} \\
D_x^2 g(x) & = -\frac{f}{f'} \cdot \frac{f'' f - (f')^2}{(f')^2} &
D_{\tau} g(x) & = -\frac{D_{\tau} f}{f'} \\
D_x D_{\tau} g(x) & = -\frac{f}{f'} \cdot \frac{f'' D_{\tau}(f) - f' D_{\tau}(f')}{(f')^2}, & 
D_x D_{z} g(x) & = -\frac{f}{f'} \cdot \frac{f'' D_{z}(f) - f' D_{z}(f')}{(f')^2}
\end{alignat*}
where on the right hand side we have omitted the argument~$\tilde{x}$ in~$f$ and its derivatives.

After changing variables and clearing denominators we find that \eqref{5dsad} is equivalent to
\begin{equation} \left(f''(x) f(x) f'(y)^2 - f''(y) f(y) f'(x)^2 \right) \cdot C + \left(f(x) f'(y) - f'(x) f(y) \right) \cdot D = 0 \label{eqn_to_vanish} \end{equation}
where we have written~$x,y$ for~$\tilde{x}, \tilde{y}$ and
\begin{align*}
C & = f'(x) f'(y) - D_z f(x) \cdot f'(y) - f'(x) D_z f(y) + f(x) D_{\tau}f(y) + D_{\tau} f(x) \cdot f(y) \\
D & = \left( f''(x) D_{z}f(x) - f'(x) D_z f'(x) \right) f'(y)^2 \\
 & + \left( f''(y) D_{z}f(y) - f'(y) D_z f'(y) \right) f'(x)^2 \\
 & - \left( f''(x) f(x) - f'(x)^2 \right) f'(y) D_{\tau} f(y) - \left( f''(x) D_{\tau} f(x) - f'(x) D_{\tau} f'(x) \right) f(y) f'(y) \\
 &  - \left( f''(y) f(y) - f'(y)^2 \right) f'(x) D_{\tau} f(x) - \left( f''(y) D_{\tau} f(y) - f'(y) D_{\tau} f'(y) \right) f(x) f'(x).
\end{align*}

Let~$\CF(x,y,z,\tau)$ be the left hand side of \eqref{eqn_to_vanish}. We need to show that~$\CF = 0$.
We will argue as in Section~\ref{subsection:proof_gen_series}. Since it is a polynomial in derivatives of Jacobi forms the function~$\CF$
is a quasi-Jacobi form of the three elliptic variables~$x,y,z$. It is of weight~$6$ and index
\[
L = \begin{pmatrix} 0 & 0 & 3/2 \\ 0 & 0 & 3/2 \\ 3/2 & 3/2 & 3/2 \end{pmatrix}.
\]
A quick check using the commutation relations \eqref{eq:comm relations 2} shows that in the algebra of such quasi-Jacobi forms we have
\[
\frac{d}{dG_2} \CF = \frac{d}{d A(x)} \CF = \frac{d}{dA(y)} \CF = 0.
\]
By a direct check (e.g. using a computer)
$\CF$ has no poles at~$y=0$
and vanishes to order~$3$ at~$y=-z$. Hence the ratio
\[ \frac{\CF(x,y)}{f(x)^3 f(y)^3}, \]
is holomorphic in~$y$.
Since by \cite[Lem.\ 6]{OPix2} it is also~$2$-periodic, we find that it is constant in~$y$.
But~$\CF$ is symmetric in~$x$ and~$y$ so it is also constant in~$x$.
By checking that the constant term vanishes we are done.
\end{proof}

\begin{remark}
In the proof we established \eqref{eq:rec2}, which is precisely Proposition~\ref{prop:recursion}.

By Proposition~\ref{prop:phimn} for all~$m, n > 0$ the function~$\phi_{m,n}$ is determined by~$\phi_m$ and~$\phi_n$.
Hence \eqref{eq:rec2} yields recursive formulas for~$\phi_m$,
and hence provides an alternative definition of the set of functions~$\phi_m$
starting from the initial condition~$\phi_1= \Theta(z)$.
For example, the case~$(n,1)$ yields
\[ \phi_{n+1} = D_z(\phi_1)\phi_n+\frac{1}{n}\phi_1D_z(\phi_n)+\sum_{i=1}^{n-1}\frac{1}{i} \phi_{i,1}\phi_{n-i}. \qedhere \]
\end{remark}

\subsection{Proof of Theorem~\ref{Thm_phi_mn}}
We need to show that for all~$n \geq 1$ we have
\[ \phi_{n,-n} -n \in \QJac_{0, n}. \]
The idea of the proof is to consider the two expressions for~$\frac{d}{dA} \phi_{m,n}$
for positive~$m,n$ given by \eqref{dAphi} and \eqref{dAphi2}.
These terms are equal for~$m > 0$, and (with minor modifications) they have natural extensions to~$m \le 0$. We will observe that these extensions are both polynomial in~$m$ (when fixing~$n$) up to the same non-polynomial correction term.
Hence they are equal for all~$m$.

Concretely, let~$n > 0$ be fixed and let~$R(m,n)$ be the right hand side of \eqref{dAphi}.
Then by Corollary~\ref{Cor_polynomiality_Anomaly} the sum of~$R(m,n)$ and
\[ -mz \delta_{m<0} \left( \frac{n}{m+n} D_{\tau}(\phi_m) \phi_n + \frac{m}{m+n} \phi_m D_{\tau}(\phi_n) \right) 
= -mz \delta_{m<0} \phi_{m,n} \]
is polynomial in~$m$. We write
\[ \widetilde{R}(m,n) = R(m,n) - mz \phi_{m,n} \delta_{m<0} \]
to denote this polynomial function.

We consider now the right hand side of \eqref{dAphi2} and we want to make sense of it for negative~$m$.
For all~$m \geq 0$, with~$m \neq n$ in the second line, define
\begin{align*}
S(m,n) & := \frac{m \cdot n}{m+n} \varphi_{m+n} + \sum_{j=1}^{m-1} \frac{m}{j} \varphi_{m-j,n} \varphi_j
+\sum_{j=1}^{n-1} \frac{n}{j} \varphi_{m,n-j} \varphi_j \\
S(-m,n) & := \frac{-m \cdot n}{-m+n} \varphi_{-m+n} + \sum_{j=1}^{m-1} \frac{m}{j} \varphi_{-m+j,n} \varphi_j
+\sum_{j=1}^{n-1} \frac{n}{j} \varphi_{-m,n-j} \varphi_j.
\end{align*}
By a direct application of Lemma~\ref{Lemma_polynomiality} the sum 
\[ \widetilde{S}(m,n)  = S(m,n) - mz \phi_{m,n} \delta_{m<0} \]
is polynomial in~$m$.

By Proposition~\ref{HAE_mnpositive} we have~$R(m,n) = S(m,n)$, hence~$\widetilde{R}(m,n) = \widetilde{S}(m,n)$ for all $m>0$.
By polynomiality in~$m$ we get~$\widetilde{R}(m,n) = \widetilde{S}(m,n)$ for all $m \neq -n$. Thus
\begin{equation} \forall m \neq -n : R(m,n) = S(m,n). \label{R=S} \end{equation}

We specialize \eqref{R=S} to~$m=-n-1$. 
Since
\[ S(-n-1,n) = -(n+1) n \phi_1 + (n+1) \phi_{-n,n} \phi_1 + \sum_{j=2}^{n} \frac{n+1}{j} \phi_{-(n+1)+j, n} \phi_j + \sum_{j=1}^{n-1} \frac{n}{j} \phi_{-(n+1), n-j} \phi_j \]
and~$\phi_1 = \Theta(z)$, the equation \eqref{R=S} yields
\[ \phi_{-n,n} - n = \frac{1}{(n+1) \Theta}\left( R(-n-1, n) - \sum_{j=2}^{n} \frac{n+1}{j} \phi_{-(n+1)+j, n} \phi_j - \sum_{j=1}^{n-1} \frac{n}{j} \phi_{-(n+1), n-j} \phi_j \right). \]
The term in the bracket on the right lies in~$\QJac_{-1,n+1/2}$ by inspection. Moreover, again by inspection it vanishes at~$z=0$. Hence it must be divisible in algebra of quasi-Jacobi forms by~$\Theta(z)$. This gives~$\phi_{-n,n} - n \in \QJac_{0,n}$.
\qed
\medskip 

\begin{remark} \label{rmk:hae}
The proof yields more information. For~$m \neq -n$ we have
$\frac{d}{dA} \phi_{m,n} = R(m,n)$ by \eqref{dAphi}.
Using that~$R(m,n) = S(m,n)$ for all~$m \neq -n$ we find the anomaly equation
\[
\frac{d}{dA} \phi_{-m,n}
=
\frac{-m \cdot n}{-m+n} \varphi_{-m+n} + \sum_{j=1}^{m-1} \frac{m}{j} \varphi_{-m+j,n} \varphi_j
+\sum_{j=1}^{n-1} \frac{n}{j} \varphi_{-m,n-j} \varphi_j
\]
where~$m, n > 0$ and~$m\neq -n$.
\end{remark}

\subsection{Holomorphic anomaly equations II} \label{subsec:HAE2}
We finally derive the precise modular properties of the functions~$\phi_{m,n}$ in terms of holomorphic anomaly equations.
\begin{prop} \label{HAEmn}
For all~$m,n \in \BZ$ we have
\begin{enumerate}
\item[\upshape(a)] ${\displaystyle \frac{d}{dG_2} \varphi_{m,n} = 2 \phi_m \phi_n}$.
\item[\upshape(b)]
$
\displaystyle \frac{d}{dA} \varphi_{m,n} =
\frac{m \cdot n}{m+n} \varphi_{m+n} + \sum_{i+j=m} \frac{|m|}{j} \varphi_{i,n} \varphi_j
+\sum_{i+j=n} \frac{|n|}{j} \varphi_{m,i} \varphi_j
$
\end{enumerate}
with the convention in {\upshape(b)} that the first term vanishes if~$m+n=0$ and that
in a sum with condition~$i+j=\ell$ (for $\ell=m$ or $\ell=n$) we sum over all positive~$i,j$ if~$\ell$ is positive, and over all negative~$i,j$ if~$\ell$ is negative.
\end{prop}
\begin{proof}
Part (a) follows from the defining differential equation \eqref{defining_diff_eqn} by applying~$d/dG_2$.
In part (b) by Proposition~\ref{HAE_mnpositive} and  Remark~\ref{rmk:hae} we only need to prove the case~$m=-n$. 
For that we restrict ourself to the region~$m<0$ and~$n>0$.
Applying~$d/dA$ to \eqref{defining_diff_eqn} yields
\[ D_z \phi_{m,n} + D_{\tau} \frac{d}{dA} \phi_{m,n} = \frac{d}{dA}\left( mn\phi_m\phi_n F+(D_{\tau}\phi_m)(D_{\tau}\phi_n) \right). \]
The right-hand side and the first term on the left-hand side are polynomial in~$m$ and~$n$ (in the considered region).
Hence~$\frac{d}{dA} \phi_{m,n}$ is polynomial in~$m,n$ up to a constant in~$q$.\footnote{There is a small subtlety here
since at first it only follows that~$\frac{d}{dA} \phi_{m,n}$ is a power series in~$z,q$ whose coefficients are polynomial in~$m,n$.
But then~$\frac{d}{dA} \phi_{m,n}$ is a quasi-Jacobi form for every~$m,n$ so that this actually has to be a power series in~$z$
with coefficients which are polynomials with coefficients quasi-modular forms (of determined weight).}
Let~$T(u,v)$ be the polynomial series such that
\[ T(m,n) = \frac{m \cdot n}{m+n} \varphi_{m+n} + \sum_{i+j=m} \frac{|m|}{j} \varphi_{i,n} \varphi_j +\sum_{i+j=n} \frac{|n|}{j} \varphi_{m,i} \varphi_j \]
for all~$m \neq -n$ in the region.
We already know~$T(m,n) = \frac{d}{dA} \phi_{m,n}$ for all~$m \neq -n$ so by the polynomiality of~$\frac{d}{dA} \phi_{m,n}$ we get
for all~$m,n$ in the region
\[ T(m,n) = \frac{d}{dA} \phi_{m,n} + c_{m,n}(z) \]
for some~$c_{m,n}(z)$ which does not depend on~$q$.
Specializing to~$m=-n$ we see
\[ \frac{d}{dA} \phi_{-n,n} + c_{-n,n}(z) = T(-n,n) = -n^2 z + \sum_{i+j=-n} \frac{n}{j} \varphi_{i,n} \varphi_j +\sum_{i+j=n} \frac{n}{j} \varphi_{-n,i} \varphi_j. \]
But~$\frac{d}{dA} \phi_{-n,n}$ is homogeneous as a quasi-Jacobi form of weight~$-1$ and index~$n$.
Hence the constant terms in~$q$ on both sides must match up and so as desired
\[ \frac{d}{dA} \phi_{-n,n} = \sum_{i+j=-n} \frac{n}{j} \varphi_{i,n} \varphi_j +\sum_{i+j=n} \frac{n}{j} \varphi_{-n,i} \varphi_j. \qedhere \]
\end{proof}

\begin{remark}
Once we know that~$\phi_{n,-n}$ is quasi-Jacobi and know its~$A$-derivative
it is not difficult to derive a recursive formula for it (ignoring that we already obtained a formula in the proof of Theorem~\ref{Thm_phi_mn}).
Indeed, consider the defining differential equation
$$D_{\tau} \phi_{m,n} = mn\phi_m\phi_n F+(D_{\tau}\phi_m)(D_{\tau}\phi_n).$$
Applying~$\frac{d}{dA}$ twice and using the commutation relations we get
\[
(|m|+|n|) \phi_{m,n} + 2 D_z \frac{d}{dA} \phi_{m,n} + D_{\tau} \left( \frac{d}{dA} \right)^2 \phi_{m,n}
=
 \left( \frac{d}{dA} \right)^2 \Big( mn\phi_m\phi_n F+(D_{\tau}\phi_m)(D_{\tau}\phi_n) \Big)
\]
Since~$(\frac{d}{dA})^i \phi_{m,n}$ is determined recursively from functions indexed by~$m',n'$ with~$m'+n'<m+n$ this yields one more formula for~$\phi_{m,n}$.
\end{remark}

\section{The classical Kaneko--Zagier equation} \label{sec:KZequations}
The differential equation introduced by Kaneko and Zagier \cite{KZ} can be characterized among quadratic differential equations
as those for which the solution space is invariant under the modular transformation for the full modular group, so that it is essentially unique \cite{KK2}.
If one however considers congruence subgroups, further differential equations of the same type have been found by Kaneko and Koike \cite{KK}.
In this section we give a general construction which takes as input a meromorphic Jacobi form of weight~$-1$
and gives as output a differential equation of Kaneko--Zagier type.
The two Kaneko--Zagier equations above and our case studied in this paper are all given by this construction.\footnote{A certain differential equation for index~$1$ Jacobi forms was studied by Kiyuna \cite{Kiyuna} and was called a Kaneko--Zagier type equation. Howver, since it is of~$4$-th order it does not fit our framework.}

\subsection{A general construction}
A general recipe to construct Kaneko--Zagier type differential equations is as follows.
Let~$g$ be a meromorphic Jacobi form of weight~$-1$.
Define
\[ E(\tau) = \frac{D_\tau g(\tau)}{g(\tau)} \quad \text{and} \quad H(\tau) = \frac{D_{\tau}^2 g(\tau) }{g(\tau)}. \]
For all~$m \geq 1$ we consider the differential equation
\[ D_{\tau}^2 g_m = m^2 H(\tau) g_m. \]

To obtain the connection to the classical presentation, we set~$m = k+1$, and consider
\[ f_k = g_{k+1} / g^{k+1} \]
which is of weight~$k$.
The corresponding differential equation for~$f_k$ reads
\[
D_{\tau}^2 f_k + 2 (k+1) E(\tau) D_{\tau} f_k +k(k+1)(E(\tau)^2-H(\tau)) f_k = 0.
\]
For this choice of~$g$ (and hence of~$E$), we define a modified Serre derivative
\[\theta_g= D_\tau+E \wt.\]
The operator~$\theta_g$ is a derivation vanishing on~$g$. Moreover, the above differential equation can be rewritten as
\begin{equation}\label{eq:kz2} \theta_g^2 f_k = H \wt (\wt + 2)  f_k. \end{equation}

We give several examples:
\begin{enumerate}
\item[(0)] In this paper we considered the case~$g(z,\tau)=\Theta(z)$. 
\item[(1)] For the classical Kaneko--Zagier equation we let
\[ g(\tau) = \frac{1}{\eta(\tau)^2} \]
and get~$H(\tau) = E_4(\tau)/144$. The operator~$\theta_g$ is the Serre derivative. 
\item[(2)] For the differential equation studied in \cite{KK} we take
\[ g(\tau) = \frac{1}{\eta(\tau) \eta(2 \tau)} \]
and get 
\[ E(\tau) =  \frac{1}{24} (E_2(\tau) + 2 E_{2}(2\tau)) \qquad
2^6 H(\tau) 
= \frac{1}{5} (E_4(\tau) + 4 E_4(2 \tau)). 
\]
The operator~$\theta_g$ matches the derivative operator of \cite[Sec.\ 2]{KK}. 
\end{enumerate}

\subsection{Recursive construction of the solutions}
Let~$f_k$ and~$f_l$ be two solutions of~(\ref{eq:kz2}) of weight~$k$ and~$l$ respectively. We write 
\begin{align*}
[f,h] & := k \theta_g(f) h - l f \theta_g(h) \\
& = k D_{\tau}(f)h-lf D_{\tau}(h)
\end{align*}
which specializes to the first Rankin-Cohen bracket on modular forms.
\begin{prop}
We have
\begin{gather*}
\theta_g[f_k,f_l] = \frac{k-l}{l+2}[f_k,\theta_g(f_l)] \\
\theta_g^2[f_k,f_l] = (k-l)(k-l-2)H[f_k,f_l]+k(k-l)f_k[f_l,H].
\end{gather*}
\end{prop}
\begin{proof}
This follows from a direct computation.
\end{proof}

\begin{cor} \label{cor:induction}
Suppose that~$[f_l,H]=0$.
Then \[[f_k,f_l]g^{2l+4} \quad \text{and} \quad [f_kg^{2k+2},f_l]g^{-2k-2}\]
are solutions of~(\ref{eq:kz2}) of weight~$k-l-2$ and~$k+l+2$ respectively. 
\end{cor}

Hence if a function~$f_{l}$ as in the corollary exists, then from any given solution we can recursively write down solutions of (\ref{eq:kz2}) with weight in the same residue class modulo~$l$.

\begin{example}
For the classical Kaneko--Zagier equation we can take~$f_l=E_4$. Then indeed~$[f_l,H]=0$, so that if~$f_k$ is a solution we have that~$[f_k,E_4]/\Delta$ is a solution of weight~$k-6$,
and~$[f_k \eta^{-4k-4}, E_4] \eta^{4k+4}$ is a solution of weight~$k+6$.
The first of these equations can also be found in \cite[Proposition 1(i)]{KK2}.
\end{example}

\begin{example}
For the Kaneko--Zagier equation in Example (2) we can take~$f_l=2 E_2(2\tau)-E_2(\tau)$. Then, indeed~$[f_l,H]=0$, so
solutions can be constructed~$4$-periodically, compare also with \cite{KK}.
\end{example}

\begin{remark}
However, in the differential equation of Example (0) considered in this paper, it turns out that the recursive structure described in Corollary~\ref{cor:induction} does not exist. To see this, suppose (for our general family of Kaneko-Zagier equations)
that there exists a solution~$f_l$ and that moreover we have~$[f_l,H]=0$. Then
the condition
$[f_l,H]=0$
is equivalent to
\[ \theta_g f_l = \frac{l}{4}\left(\frac{D_{\tau}^3g}{D_{\tau}^2g}+3\frac{D_{\tau}g}{g}\right)f_l. \]
Applying~$\theta_g$ to this equation and using the differential equation for the left hand side, we obtain
\[ 
16(l+2)\frac{D_{\tau}^2g}{g}
=
4\frac{D_{\tau}^4g}{D_{\tau}^2g}+12\frac{D_{\tau}^2g}{g}+(l-4)\left(\frac{D_{\tau}^3g}{D_{\tau}^2g}\right)^2+2(3l+4)\frac{D_{\tau}^3g}{D_{\tau}^2g}\frac{D_{\tau}g}{g}+3(3l+4)\left(\frac{D_{\tau}g}{g}\right)^2.\]
For~$g=\eta^{-2}$ this equation is only satisfied if~$l=4$, and for~$g=\left(\eta(\tau)\eta(2\tau)\right)^{-1}$ only if~$l=2$. However, for~$g=\Theta(z)$ this equation is never satisfied, so Corollary~\ref{cor:induction} cannot be applied.
\end{remark}

\end{document}